\documentclass[reqno,oneside,12pt,letterpaper]{amsart}

\usepackage{dwpreamble}
\usepackage{dwcommands}

\renewcommand{\geq}{\geqslant}
\renewcommand{\leq}{\leqslant}

\begin{document}

\title[The wave resolvent for compactly supported perturbations of Minkowski space]{\large{The wave resolvent for compactly supported \\ perturbations  of static spacetimes}\smallskip}

\author{}
\author[]{\normalsize Micha{\l} \textsc{Wrochna} \& Ruben \textsc{Zeitoun}}
\address{}
\address{Laboratoire AGM (UMR 8088, CNRS), CY Cergy Paris Universit\'e, France, \phantom{CY Cergy Paris Universit\'e, 2 av.~Adolphe Chauvin} \phantom{O} Freiburg Institute for Advanced Studies (Frias), University of Freiburg, Germany}

\email{michal.wrochna@cyu.fr}
\address{Laboratoire AGM (UMR 8088, CNRS), CY Cergy Paris Universit\'e, France}
\email{ruben.zeitoun@ens-lyon.fr}

\begin{abstract} In this note, we consider the wave operator  $\square_g$ in the case of globally hyperbolic, compactly supported perturbations of static spacetimes.  We give an elementary proof of the essential self-adjointness of $\square_g$ and of uniform microlocal estimates for the resolvent  in this setting.  This  provides a model for studying Lorentzian spectral zeta functions which is particularly simple, yet sufficiently general for locally deriving  Einstein equations from a spectral Lagrangian action.
\end{abstract}

\maketitle

\section{Introduction}

\subsection{Motivation} Let $P=\square_g$ be  the wave operator on a Lorentzian manifold $(M,g)$. It was shown by Vasy \cite{vasyessential} that if $(M,g)$ is a \emph{non-trapping Lorentzian scattering space} then $\square_g$ is essentially self-adjoint in the sense of the canonical $L^2(M,g)$ space. This result was then generalized by Nakamura--Taira \cite{nakamurataira,Nakamura2022,Nakamura2022a}  to \emph{long-range perturbations of Minkowski space}, higher order operators and \emph{asymptotically static spacetimes} with compact Cauchy surface. In consequence, in each of these settings one can define complex powers $(\square_g-i \varepsilon)^{-\alpha}$  by functional calculus for all $\varepsilon>0$.

In the first situation, it was shown in \cite{Dang2020} that under the extra hypothesis that $n\geq 4$ is even and $(M,g)$ is \emph{globally hyperbolic}, the Schwartz kernel of $(\square_g-i \varepsilon)^{-\cv}$ has for $\Re\cv>\frac{n}{2}$ a well-defined on-diagonal restriction $(\square_g-i \varepsilon)^{-\cv}(x,x)$, which extends to a meromorphic function of $\cv\in\cc$ (called the \emph{Lorentzian spectral zeta function density}).   Furthermore, the residues can be expressed in terms of the metric $g$, in particular: 
\beq\label{thepole}
\lim_{\varepsilon\to 0^+}\res_{\cv=\frac{n}{2}-1} \left(\square_g- i \varepsilon\right)^{-\cv}(x,x)=  \frac{R_g(x)}{{i}6(4\pi)^{\n2} \Gamma\big(\frac{n}{2}-1\big) }, 
\eeq
where $R_g(x)$ is the scalar curvature at $x\in M$.  Since the variational principle $\delta_g R_g=0$ is equivalent to vacuum Einstein equations and the l.h.s.~refers to spectral theory, this gives a \emph{spectral action} (or strictly speaking, Lagrangian) for gravity.

The proofs of essential self-adjointness and formula \eqref{thepole} rely on \emph{microlocal radial estimates} \cite{melrosered,Vasy2013,GHV,vasygrenoble,vasyessential}, which are nowadays broadly used in hyperbolic problems. The non-expert reader might however not be familiar with the  required formalism, nor with the various technical issues that arise from the combination of microlocal and global aspects (even the definition of non-trapping Lorentzian scattering spaces requires some familiarity). 

In this note, our objective  is to present a much simpler model in which it is possible to give more elementary proofs.  This is motivated first of all by pedagogical reasons, but also by the need of having a toy model for testing various ideas that go beyond formula \eqref{thepole}.

The easiest case is without doubt the class of \emph{ultra-static} spacetimes $(M,g)$ (Min\-kow\-ski space being the primary example). In this situation,  the wave operator $\square_g$ is of the form  $\p_t^2-\Delta_h$ for some $t$-independent Riemannian metric $h$. Essential self-adjointness is then almost immediate (provided that $\Delta_h$ is essentially self-adjoint), and it can also be easily proved for  more general \emph{static} metrics (see  Derezi\'nski--Siemssen \cite{derezinski}) in which case there are extra multiplication operators in the expression for $\square_g$. The proof of \eqref{thepole}  simplifies   as well, at least for ultra-static metrics \cite{Dang2020}. However, this type of assumptions is in practice too restrictive because it narrows down the allowed metric variations to time-independent ones. 

This leads us to consider   \emph{compactly supported perturbations} of static spacetimes.  Such perturbations are indeed sufficient for formulating a variational principle and for the purpose of illustrating propagation phenomena arising in greater generality. On the other hand, the assumption that the perturbation has compact support  allows us to largely bypass the asymptotic analysis,  and we can give proofs based almost exclusively on variants of Hörmander's classical propagation of singularities theorem.

\subsection{Main result and sketch of proof} More precisely, let $(Y,h)$ be a Riemannian  metric of dimension $n-1$ (where $n\geq 2$), and let $(M,g_0)$ be $M=\rr\times Y$ equipped with a Lorentzian metric of the form
$$
\bea
g_0&= \beta \,dt^2 - h = \beta^2(y) dt^2-h_{ij}(y) dy^{i} dy^j,
\eea
$$
for some positive $\beta\in \cf(Y)$. A metric of this form is called \emph{static}, or more precisely, \emph{standard static} (see e.g.~\cite{Sanchez2005} for more remarks on the terminology). In the special case $\beta=1$ the metric is said to be \emph{ultra-static}; the latter is the natural Lorentzian analogue of a Riemannian product-type metric.

 Let $g$ be another smooth Lorentzian metric on $M$. We make the following assumptions.

\begin{assumption}\label{hypothesis} We assume that:
\ben
\item the Riemannian manifold $(Y,h)$ is complete;
\item  $g$ is a compactly supported perturbation of $g_0$, i.e.
\beq
\supp (g- g_0) \mbox{ is compact};
\eeq
\item there exists a constant $C>0$ such that $C < \beta(y)< C^{-1}$ for all $y\in Y$;
\item  $(M,g_0)$ and $(M,g)$ are globally hyperbolic spacetimes.
\een
\end{assumption}

We recall that a Lorentzian manifold $(M,g)$ is a \emph{globally hyperbolic spacetime} if it is time oriented and  there exists a  Cauchy surface, i.e.~a  closed subset of $M$ intersected exactly once by each maximally extended time-like curve. We remark that when $(M,g_0)$ is (for instance)   Minkowski space, then global hyperbolicity of the perturbed spacetime $(M,g)$ is equivalent to a non-trapping condition, see \cite[Prop.~4.3]{GWfeynman}.

Let $\square_g$ be the wave operator, or d'Alembertian on $(M,g)$, i.e.~the   Laplace--Beltrami operator for the Lorentzian metric $g$. More explicitly, denoting $\module{g}=\module{\det g}$ for brevity, we have
$$
\bea
\square_g=\module{g(x)}^{-\frac{1}{2}}\partial_{x^j} \module{g(x)}^{\frac{1}{2}}g^{jk}(x)\partial_{x^k}
\eea
$$
where we sum over repeated indices.  In this setting, we prove the following result.

\bet\label{theorem1}  Assume Hypothesis \ref{hypothesis}. Then the wave operator 
$\square_g$ is essentially self-adjoint on $\ccf$ in $L^2(M,g)$.
\eet

Furthermore, we show    \emph{uniform microlocal resolvent estimates} for the wave operator $\square_g$ (strictly speaking, its closure). In \cite{Dang2020} they are a key ingredient in the analysis of complex powers of $\square_g$. We give here an analogue in our setting.

\bet\label{theorem2} Assume Hypothesis \ref{hypothesis}.  Then the wave resolvent $(\square_g-z)^{-1}$ has Feynman wavefront set. More precisely, let $s\in \rr$, $\varepsilon>0$ and $\theta\in \open{0,\pi/2}$.  Then for $\module{\arg  z - \pi/2} < \theta$, $\module{z}\geq \varepsilon$, the uniform operator wavefront set of  $(\square_g-z)^{-1}$ of order $s$ and weight $\bra z\ket^{-\12}$  (see Definition \ref{defrrr})   satisfies 
$$
\wfl{12}  \big( ( \square_g-z)^{-1} \big)\subset  \Lambda,
$$
where $\Lambda$ is the (primed) Feynman wavefront set (see Definition \ref{deffey}).
\eet

This type of estimates is used in \cite{Dang2020} to show that the resolvent and complex powers of $\square_g$ are sufficiently well approximated by a \emph{Hadamard parametrix},  which in turn can be used to extract the scalar curvature $R_g$ (see \cite{Dang2022} for a brief review). That subsequent analysis is completely general, and so by combining  Theorem \ref{theorem2} with  \cite[\S\S4--8]{Dang2020}    we obtain the following result (see also \cite{Dang2021} for further consequences).

\begin{corollary} Assume Hypothesis \ref{hypothesis}. Then the identity \eqref{thepole} holds true in even dimension $n=\dim M\geq 4$.
\end{corollary}

  We remark that while our assumptions are certainly restrictive,  our results are not exclusively special cases of \cite{vasyessential,nakamurataira,Dang2020,Nakamura2022,Nakamura2022a} because we allow for more general behaviour in the spatial directions. Together with the recent work \cite{Nakamura2022a}, this provides further evidence for Derezi\'nski's conjecture \cite{Derezinski2019} that essential self-adjointness may hold true on a large class of \emph{asymptotically static} spacetimes  (with possibly general behaviour in the spatial directions). We  conjecture that the statement of Theorem \ref{theorem2} would remain valid as well.
  
 \subsection{Structure of paper}  Essential self-adjointness, i.e.~Theorem \ref{theorem1}, is proved in  \sec{section1}, preceded by various preliminaries on propagation of singularities.  Theorem \ref{theorem2} is proved in \sec{section2}; that       section also contains the necessary background on operator wavefront sets.

\section{Essential self-adjointness}  \label{section1}

\subsection{Preliminaries on  self-adjointness} \label{sec2.1} Let us first consider the ultra-static case $\beta=1$. Let
$$
P_0=\p_{t}^2 - \Delta_h
$$ 
be the unperturbed wave operator, i.e.~the wave operator on the static spacetime $(M,g_0)$. In that case there is an  argument that gives its essential self-adjointness immediately.

 \begin{lemma} $P_0$ is essentially self-adjoint  on $C_{\rm c}^\infty(M)$ in $L^2(M,g_0)$.
 \end{lemma} 
  \begin{proof} We quote  the argument from \cite{derezinski} for the reader's convenience. We know that $D_t^2$ is essentially self-adjoint on $C^\infty_\c(\mathbb R)$  in $L^2(\rr)$, and $\Delta_h$ is essentially self-adjoint on $C^\infty_\c(Y)$ in $L^2(Y,h)$ \cite{Chernoff1973}.  
   Therefore by \cite[\S VIII.10]{ReedSimonI}, $P_0=-D_t^2\otimes \one -\one\otimes \Delta_h$ is essentially self-adjoint on the algebraic tensor product of $ C^\infty_\c(\mathbb R)$ with $C^\infty_\c(Y)$, which is dense in $C^\infty_\c(M)$ in  $L^2(M,dt^2+h)=L^2(M,dt^2-h)=L^2(M,g_0)$. 
  \end{proof}

Denoting also by $P_0$ the closure, the resolvent $(P_0-z)^{-1}$ exists  for $z\in \cc\setminus \rr$.   

Let us denote by $L_0$ the closure of minus the Laplace--Beltrami operator $\p_t^2 + \Delta_h$ on the complete Riemannian metric $dt^2+h$. We use it to introduce a global Sobolev space of order $s\in\rr$:
$$
H^s(M)\defeq (\one+L_0)^{-\frac{s}{2}} L^2(M,g_0),
$$ 
  i.e.~the norm is given by  $\norm{ u }_{H^s}= \norm{  (\one+ L_0 )^{\frac{s}{2}} u   }_{L^2}$ in terms of the norm of $L^2(M,g_0)$. We will also frequently write $L^2(M)$ instead of $L^2(M,g_0)$ for the sake of brevity. Since  $P_0$ commutes with   $L_0$, for all $m\in \rr$ we can extend  the resolvent to an operator $(P_0-z)^{-1}\in B( H^{m}(M), H^{m}(M))$ which satisfies $(P_0-z) (P_0-z)^{-1}=\one$ on $H^m(M)$. By a direct computation one can check the formula
\beq\label{eq:reso}
\big((P_0-z)^{-1} f\big)(t) = -\12   \int_{\rr} \frac{e^{-i\module{t-s}\sqrt{-\Delta_h -z}}}{\sqrt{-\Delta_h -z} }  f(s) ds,
\eeq
for $\Im z >0$ and $f\in L^2(M)$, where the r.h.s.~is defined using Fourier transform and functional calculus.  

Let us now focus on the wave operator $\square_g$ for the perturbed metric $g$. Let $U: L^2(M,g_0)\to L^2(M,g)$ be the multiplication operator by $\module{g}^{-\frac{1}{4}}\module{g_0}^{\frac{1}{4}}$, and let
$$
P:=  U^*  \square_g U.
$$
Then,  $\supp (P-P_0)$ is compact, and  since $U$ is bounded and boundedly invertible, essential self-adjointness of $\square_g$ in $L^2(M,g)$ is equivalent to essential self-adjointness of $P$ in $L^2(M,g_0)$. 

  Recall that the standard criterion for essential self-adjointness says that it suffices to show the implication
\beq
\forall u\in L^2(M) \mbox{ s.t. } (P \pm i ) u = 0, \ u=0,
\eeq
where $(P \pm i ) u=0$ is meant in the sense of distributions. While the two conditions with different signs are needed, they are largely analogous so we will only consider the `$-$' case. 

The basic argument consists in writing for all $u\in L^2(M)$  such that $(P-i)u=0$,
$$
2 i \norm{u}^2_{L^2} = ( P u | u)_{L^2} - ( u | Pu  )_{L^2}.
$$
If $u \in H^2(M)$, by integration by parts the latter expression vanishes, and we conclude in that case $u=0$.  For this reason it suffices to prove 
 \beq\label{criterion}
 \forall u\in L^2(M) \mbox{ s.t. } (P \pm i ) u = 0, \ u\in H^2(M).
 \eeq
 
 As shown by  Nakamura--Taira \cite{nakamurataira}, in the case of compactly supported perturbations, global aspects can be dealt with relatively easily.
   
We denote by $\Psi^m(M)$ the set of pseudo-differential operators of order $m\in \rr$ on $M$ (in the sense of the general pseudo-differential calculus on manifolds, see e.g.~\cite[\S4.3]{shubin}). 

\begin{proposition} \label{nakamura} Assume $\beta=1$. Let  $k\in\nn_{\geq 0}$   and suppose $u\in L^2(M)\cap H^{k+1}_\loc(M)$  satisfies $(P-i)u=0$. Then $u\in H^k(M)$.
\end{proposition}
\begin{proof} The proof of {\cite[Prop.~C.1]{nakamurataira}} applies verbatim to our case; we repeat it for the reader's convenience. Set $N_\varepsilon=(\one+L_0)^{\12}(\one+\varepsilon L_0)^{-\12}$, $\varepsilon\geq 0$. For $\varepsilon>0$, $N_\varepsilon\in \Psi^0(M)\cap B(L^2(M))$, hence $N_\varepsilon^{2k} u\in L^2(M)\cap H^{k+1}_\loc(M)$. Let $\psi\in \cf(M)$ be such that $\psi=0$ in a neighborhood of $\supp(P-P_0)$ and $\psi=1$ on the complement of some compact set.

Then,
\beq\label{pzp}
P_0(\psi u)=P(\psi u)= \psi Pu+[P,\psi]u=-i\psi u+Bu,
\eeq
where $B:=[P,\psi]$ is of order $1$ and has compactly supported coefficients. The latter implies $Bu\in H^k(M)$, so by \eqref{pzp} we get $P_0(\psi u)\in L^2(M)$. 
We can now compute 
\begin{equation}
\bea
2i\Im (N_\varepsilon^{2k}(\psi u)|P_0(\psi u))_{L^2}&= 2i \Im  (N_\varepsilon^{2k}(\psi u)|-i\psi u+Bu)_{L^2}\\
 &= 2\Vert N_\varepsilon^{k}(\psi u)\Vert_{L^2}^2+2i \Im (N_\varepsilon^{2k}(\psi u) | Bu)_{L^2}.
\eea
\end{equation}
On the other hand,  $[N_\varepsilon,P_0]=0$, $N_\varepsilon$ is bounded and $P_0(\psi u)\in L^2(M)\cap H^{k}_\loc(M)$, so $P_0(N_\varepsilon^{2k}(\psi u))=N_\varepsilon^{2k}(P_0(\psi u))\in L^2(M)\cap H^{k}_\loc(M)$. In consequence, 
\begin{equation}
2i \Im (N_\varepsilon^{2k}(\psi u)|P_0(\psi u))_{L^2}=(N_\varepsilon^{2k}(\psi u)|P_0(\psi u))_{L^2}-(P_0(\psi u)|N_\varepsilon^{2k}(\psi u))_{L^2}=0.
\end{equation}
Thus, we have
\begin{equation}
\Vert N_\varepsilon^{k}(\psi u)\Vert_{L^2}^2=\module{\Im  (N_\varepsilon^{2k}(\psi u),Bu)_{L^2}} \leq \Vert N_\varepsilon^{k}(\psi u)\Vert_{L^2}\Vert   N_\varepsilon^{k}B u\Vert_{L^2},
\end{equation}
hence $\Vert N_\varepsilon^{k}(\psi u)\Vert_{L^2} \leq \Vert   N_\varepsilon^{k}Bu\Vert_{L^2}$. 
Since $L_0\geq 0$, $N_\varepsilon \leq N_{\varepsilon'}$ for $\varepsilon'<\varepsilon$.  Moreover, $N^{k}_0 B u\in L^2(M)$ since $Bu\in H^k(M)$. Therefore, by monotone convergence, as $\varepsilon\to 0^+$ we get $\Vert N^{k}_0(\psi u)\Vert_{L^2} \leq \Vert   N^{k}_0Bu\Vert_{L^2}<+\infty$. Since $N_0=\bra L_0 \ket$,
 this implies $\psi u\in H^k(M)$ as claimed.
\end{proof}

\subsection{Preliminaries on microlocal analysis} In view of  Proposition \ref{nakamura} we are left with  the task of proving sufficient local regularity  of $L^2$ solutions of $(P-i)u=0$. To that end we will need several basic notions from microlocal analysis.

 We will  write  $(x;\xi)=(t,y;\tau,\eta)$ for points in $T^*M$
 and $\zero$ for the zero section.
Let $p(x;\xi)$ be the principal symbol of $P$, and let  $\Sigma= p^{-1}(\{ 0\})$ be its characteristic set. It splits into two connected components, $\Sigma=\Sigma^+ \cup \Sigma^-$, where the sign convention is fixed by saying that in the special case when $p(x;\xi)=p_0(x;\xi)=-\tau^2 + \eta^2$, $\Sigma^\pm$ equals
$$
 \Sigma^\pm_0 = \{ (t,y;\tau,\eta)\in T^*M\setminus \zero \,\st\,     \tau =  \pm  \module{\eta}\}.
$$

Let us recall that \emph{bicharacteristics} are integral curves of the \emph{Hamilton vector field} $H_p$ of $p$, defined in terms of the Poisson bracket by $H_p=\{ p,\cdot \}$. For a pair of points $(x_i;\xi_i)\in T^*M \setminus \zero$, $i=1,2$, we write $(x_1;\xi_1) {\sim} (x_2;\xi_2)$ if $(x_1;\xi_1)\in \Sigma$  and $(x_2;\xi_2)$ can be joined from $(x_1;\xi_1)$ by a bicharacteristic in $\Sigma$. 

Recall that given $u\in  \cD'(M)$, its \emph{Sobolev wavefront set} $\wf^{(s)}(u)$ of order $s\in \rr$  is defined as follows: $(x;\xi)\in T^*M \setminus \zero$ is \emph{not} in $\wf^{(s)}(u)$ if and only if there exists a properly supported $B\in \Psi^0(M)$ (or equivalently, $B\in \Psi^m(M)$ for some $m\in\rr$) such that $B u \in H^s_\loc(M)$  (resp.~$B u \in H^{s-m}_\loc(M)$). 

Let us recall a special case of Hörmander's classical propagation of singularities theorem for real principal type operators ($P-z$ is of real principal type by global hyperbolicity of $(M,g)$, see e.g.~\cite[Prop.~4.3]{radzikowski1996micro}), formulated here in terms of the Sobolev wavefront set.  

\bep[{\cite[\S6.3]{DH}}]  Let $z\in \cc$ and suppose  $u\in \cD'(M)$ satisfies  $f:=(P-z) u\in H^{s-1}_\loc(M)$. If $(x;\xi) \in \wf^{(s)}(u)$, then $(x;\xi) \in \Sigma$, and furthermore $(x',\xi')\in \wf^{(s)}(u)$ for all  $(x',\xi')\in T^*M\setminus \zero$ such that $(x,\xi)\sim (x',\xi')$. 
\eep

Strictly speaking, the basic statement that $(x;\xi)\in\Sigma$ is   referred  to as \emph{microlocal elliptic regularity} or the \emph{elliptic estimate}, as it can indeed be written in the form of a uniform estimate.

\subsection{Proof of local regularity} \label{ss:locreg} 

 Let $V=P-P_0$. By hypothesis, $V$ is a second order differential operator with compactly supported coefficients.  Let $T>0$ be large enough so that $\supp V \subset [-T,T]\times Y$.

We start by showing a key lemma about microlocal regularity for large times. Although in the proof of essential self-adjointness we will only need a particular case with fixed $z$ and $f=0$, the general statement will be useful in the next section. For further reference the lemma is stated for general $P$ obtained with compactly supported perturbations of $P_0$.

\bel\label{lelemme} Let $P$ be a second order differential operator  such that $V=P-P_0$  has compactly supported coefficients. Assume $\beta=1$.  Let $(x_1;\xi_1)=(t_1,y_1;\tau_1,\eta_1)\in \Sigma^\pm$ be such that $\pm t_1 >T$.  Then  for $\Im z \geq \varepsilon>0$,   there exists a bounded family of properly supported pseudo-differential operators $B_\pm(z)\in \Psi^0(M)$, each elliptic at $(x_1;\xi_1)$ and such that for all 
$u\in L^2(M)$ satisfying $f\defeq  (P-z)u \in L^2_\c(M)$, 
\beq\label{reg1}
B_\pm(z) (u - (P_0-z)^{-1} f)=0.
\eeq
If in addition $\supp f \subset [-T_-,T_+]\times Y$ for some $T_+,T_->0$ and $\pm t_1> \pm T_\pm$ then
\beq\label{reg2}
B_\pm(z) u=0.
\eeq
\eel
\begin{proof}
For all $u\in  L^2(M)$, if $f=(P-z) u\in L^2_\c(M)$ then  $(P_0-z)u=f-V u$ as elements of $H^{-2}_{\rm loc}(M)$, and 
 \beq \label{ulo}
 u -(P_0-z)^{-1}f   =- (P_0-z)^{-1}   V u.
 \eeq 
Let  $A(z)=({-\Delta_h - z})^{1/2}$. Then $A(z)\in \Psi^1(M)$, and its principal symbol is $|\eta|^\12_h$ (cf.~the last paragraph in the proof of Lemma \ref{lelemme2}).  Setting $v= (\one \otimes A(z)^{-1}) V u$ and using the formula \eqref{eq:reso} for $(P_0-z)^{-1}$, extended to elements $H^{-2}_{\rm loc}(M)$ supported in a finite time interval, we obtain 
$$
 \big(u -(P_0-z)^{-1}f \big) (t) =    \12   \int_{\rr} {e^{-i\module{t-s} A(z)} }  v(s) ds.
$$
Since $\supp v \subset [-T,T]\times Y$, this implies that
\beq\label{eq:rrr}
 \big(u -(P_0-z)^{-1}f \big) (t)=    \12  e^{\mp i t A(z)}   \int_{\rr} {e^{\pm i s A(z)} } v (s) ds \,  \mbox{ for } \pm t >  T.
\eeq
 In consequence, 
$$
(D_t \pm  A(z))  \big(u -(P_0-z)^{-1}f \big) (t) = 0 \,  \mbox{ for } \pm t >  T.
$$
If in addition $\supp f \subset [T_-,T_+]\times Y$, then we can represent $(P_0-z)^{-1}f $ similarly as the r.h.s.~of \eqref{eq:rrr}. Hence $(D_t \pm  A(z))  (P_0-z)^{-1}f =0$ for $\pm t > \pm T_\pm$ and we conclude 
\beq\label{eq:ac}
(D_t \pm  A(z))  u(t) = 0 \,  \mbox{ for } \pm t >  \max\{T,\pm T_\pm\}.
\eeq
Now, let $(x_1;\xi_1)=(t_1,y_1;\tau_1,\eta_1)\in\Sigma^\pm$ be such that $\pm t_1 >T$. Although $D_t \pm  A(z)=(D_t \otimes \one) \pm (\one \otimes A(z))$ is not a pseudo-differential operator in $\Psi^1(M)$ (instead, it is in some larger class with rather bad properties), there exists  $B_0\in \Psi^0(M)$ properly supported such that  
$$
B_{\pm,1}(z):=B_0  (D_t \pm  A(z) ) \in \Psi^1(M),
$$
and such that $B_{\pm,1}(z)$ is elliptic at $(x_1;\xi_1)$. In fact, since $(t_1,y_1;\tau_1,\eta_1)\in \Sigma$, we have $\eta_1\neq 0$ (as $\eta_1=0$ would  imply  $\tau_1=0$), so we can choose $B_0=0$ microlocally in a conic neighborhood of  $\{ \eta_1=0 \}$ and $B_0=1$ in a punctured neighborhood of it, see \cite[Thm.~18.1.35]{HormanderIII}, cf.~the proof of \cite[(3), Prop.~6.8]{GS}. Finally, by composing  $B_{\pm,1}(z)$ with a suitable family $C(z)\in \Psi^{-1}(M)$, vanishing for $\pm t \leq T$ (resp.~for $\pm t \leq \max\{ T, \pm T_\pm\}$), we obtain $B_\pm(z):=C(z) B_{\pm,1}(z)$ with the desired uniformity in $\Psi^0(M)$ and satisfying \eqref{reg1} (resp.~\eqref{reg2}).
\end{proof}

\begin{remark} Lemma \ref{lelemme} is a microlocal regularity statement at large, but \emph{finite} times, and then our next step will be to deduce a corresponding statement for arbitrary times  by Hörmander's propagation of singularity theorem. In more general situations, one needs to start with a regularity statement at \emph{infinite} times, which motivates the use of radial propagation estimates or related methods \cite{vasyessential,nakamurataira,Nakamura2022,Nakamura2022a}. In these settings, the asymptotic analogues of the two conditions \eqref{eq:ac} can be thought as boundary conditions at infinity \cite{GWfeynman}:  these were shown by Taira to be satisfied in the case of the wave resolvent on asymptotically Minkowski spacetimes \cite{Taira2020a}.
\end{remark}

\bep \label{key1} 
 Assume $\beta=1$, and suppose $u\in L^2(M)$ satisfies $(P-i )u=0$. Then $u\in \cf(M)$.
\eep  
\begin{proof}  
For any $(x;\xi)\in \Sigma^\pm\cap \{  \pm t >  T \}$ we use Lemma \ref{lelemme} with $z=i$ and $f=0$, which gives existence of $B_\pm\in \Psi^0(M)$ elliptic at $(x;\xi)$ such that $B_\pm u =0$. Thus, $(x;\xi)\notin \wf^{(s)}(u)$ for all $s\in \rr$.  We conclude
$$
\wf^{(s)}(u)\cap \Sigma^\pm \cap \{  \pm t >  T \} = \emptyset.
$$
By propagation of singularities, this implies $\wf^{(s)}(u)\cap \Sigma^\pm=\emptyset$. Since $\wf^{(s)}(u)\subset \Sigma= \Sigma^+\cup \Sigma^-$ we deduce immediately $\wf^{(s)}(u)=\emptyset$ for all $s\in \rr$, hence $u\in \cf(M)$.
\end{proof}

Proposition \ref{key1} combined with Proposition \ref{nakamura}   implies  \eqref{criterion}. This concludes the proof of essential self-adjointness of $P$, hence the self-adjointness of $\square_g$ stated in Theorem \ref{theorem1} in the case  $\beta=1$. 

\subsection{Generalization to static spacetimes} Let us now discuss the adaptations needed to prove the essential self-adjointness in the case when the spacetime is not necessarily ultra-static, i.e.~when $\beta$ is not necessarily $1$. 

The unperturbed wave operator is then
$$
P_0=\beta^{-1}\p_{t}^2 - \Delta_h. 
$$
Thanks to the assumption $C < \beta< C^{-1}$, the multiplication operator $\beta$ is bounded with bounded inverse.  Let 
$$
\tilde P_0 = \beta^\12  P_0 \beta^\12, \quad \tilde P = \beta^\12 P \beta^\12, \quad \tilde \Delta_h =  \beta^\12 \Delta_h \beta^\12.
$$ 
Then, as observed in \cite{derezinski}, essential self-adjointness of $P$ is equivalent to essential self-adjointness of $\tilde P$. Furthermore, $$\tilde P_0=\p_{t}^2 - \tilde\Delta_h$$ with  $\tilde \Delta_h$ essentially self-adjoint, and the coefficients of $\tilde P - \tilde P_0$ are compactly supported. Therefore, we can repeat the arguments from \secs{sec2.1}{ss:locreg} to show the essential self-adjointness of $\tilde P_0$ and $\tilde P$, and hence of $P$.

 This concludes the proof of Theorem \ref{theorem1}.

In the next section we will be interested in the resolvent $(P-z)^{-1}$, which is not related in a straightforward way with the resolvent of $(\tilde P-z)^{-1}$. For this reason we will need a more direct approach. The key fact is that Lemma \ref{lelemme} remains valid for $P$ with $\beta \neq 1$, as shown below. 

\begin{lemma}\label{lelemme2} The assertion of Lemma \ref{lelemme} holds true for $P_0$ and $P$ without the assumption $\beta=1$.
\end{lemma}
\begin{proof} Let $\Im z \geq \varepsilon>0$. In comparison with the case $\beta=1$, the main difference  is that the formula for the unperturbed resolvent $(P_0-z)^{-1}$ needs to be modified. We have indeed
\beq\label{eq:rel}
(P_0-z)^{-1}=\beta^\12 (\tilde P_0- z \beta )^{-1} \beta^\12,
\eeq
provided that we check that $\tilde P_0- z \beta=\p_{t}^2 - \tilde\Delta_h-z \beta$ is boundedly  invertible.

Let us first define $$L(z)\defeq i(- \tilde\Delta_h-z \beta), \mbox{ with domain }  \Dom L(z)\defeq\Dom (-\tilde\Delta_h).$$ Since $\beta$ is bounded, the operator $L(z)$ is closed. Furthermore,
\beq \label{eq:o}
\Re( u | L(z) u) = (\Im z) (u | \beta u) \geq \frac{1}{2}C^{-1}\varepsilon \| u \|_{L^2}, \ \ u \in\Dom L(z), 
\eeq
so $L(z)$ is m-accretive and $0\notin \sp(L(z))$. By \cite[\S3, Thm.~3.35]{K1}, $L(z)$ has a unique m-accretive square root $L(z)^\12$, which in addition is sectorial of angle $\frac{\pi}{4}$ and satisfies $0\notin \sp(L(z)^\12)$. It follows that if we set
\beq\label{er:ff}
A(z) \defeq e^{-i \pi/4} L(z)^\12,
\eeq
 then $0\notin \sp(A(z))$ and moreover,  $i A(z)$ is m-accretive. In consequence, $-i A(z)$ is the generator of a strongly continuous contraction semigroup denoted by $\rr_+ \ni t  \mapsto e^{-it A(z)}\in B(L^2(M))$. Let now
$$
(R(z)f)(t)=  \int_{\rr} {e^{-i\module{t-s}A(z)}}  A(z)^{-1} f(s) ds
$$
for $f\in L^2_{\rm c}(M)$. Since $(\one  \otimes A(z)^{-1})f\in L^2_{\rm c}(\rr; \Dom A(z))$, standard semigroup theory applies, and we get easily $R(z)f\in C^0(\rr;\Dom A(z))$, in particular $R(z)f$ is a distribution.  In the sense of distributions,
\beq\label{eq:reso2}
\bea
(\tilde P_0-z \beta) R(z) f&=(\p_t^2 + A(z)^2) R(z) f\\
&=(\p_t  - i A(z))(\p_t  + i A(z)) R(z) f = f
\eea
\eeq
for all $f\in L^2_\c(M)$. On the other hand, by a computation analogous to \eqref{eq:o} we obtain that the operator  $i(\tilde P_0-z \beta)$ with domain $\Dom \tilde P_0$ is m-accretive and boundedly invertible.  By applying its inverse to both sides of \eqref{eq:reso2} we obtain $(\tilde P_0-z \beta)^{-1} = R(z)$ on $L^2_\c(M)$. We conclude that $(P_0-z)^{-1}=\beta^\12 (\tilde P_0-z \beta)^{-1}\beta^\12 = \beta^\12 R(z) \beta^\12$ on $L^2_\c(M)$. In summary,
$$
((P_0-z)^{-1}f)(t)=  \beta^\12\int_{\rr} {e^{-i\module{t-s}A(z)}}  A(z)^{-1} \beta^\12 f(s) ds, \ \ f\in L^2_\c(M).
$$
From that point on we can repeat the proof  of Lemma \ref{lelemme} with $(D_t \pm  A(z))$ replaced by $(D_t \pm  A(z))\beta^{-\12}$, where $A(z)$ is defined in \eqref{er:ff}.  

This requires us to check that    $A(z)\in \Psi^1(M)$. In fact, we can show  in analogy to the proof of \cite[Prop.~4.7]{Helffer2005} that the resolvent $(L(z)-\lambda)^{-1}$ of $L(z)$ satisfies a variant of the Beals criterion in global Sobolev spaces defined using $-\tilde{\Delta}_h$. Then, for all $\chi_1,\chi_2\in \ccf$, $\chi_1 A(z) \chi_2$ can be expressed as an integral of $\chi_1 (L(z)-\lambda)^{-1}) \chi_2$ (see the proof of \cite[\S3, Thm.~3.35]{K1}). By repeating the arguments in the proof of \cite[Thm.~4.8]{Helffer2005} (with all relevant formulas multiplied by $\chi_1$ and $\chi_2$) we conclude that $A(z)\in \Psi^1(M)$, and  its principal symbol equals $\sigma_{\rm pr}\big( A(z)\big)(y;\eta)=|\eta|^\12_h(y)$.
\end{proof}


 
\section{Uniform microlocal estimates} \label{section2}
 
\subsection{Uniform wavefront set}  Throughout this section we  will write $P=\square_g$ (rather than $P=U^* \square_g U$).

 We start by introducing the  uniform wavefront set  which appears in the formulation of Theorem \ref{theorem2}. 

\begin{definition}\label{defrrr} Let $Z\subset \cc$ and suppose $\{ G(z)\}_{z\in Z}$  is for all $m\in\rr$ a bounded family of operators in $B(H^m_\c(M),H^m_\loc(M))$.
The \emph{uniform operator wavefront set of order $s\in\rr$ and weight $\bra z\ket^{-\12}$} of $\{ G(z)\}_{z\in Z}$ is the set 
\beq\label{eq:wfs}
\wfl{12}\big( G(z) \big)\subset (T^*M\setminus\zero)\times (T^*M\setminus\zero)
\eeq
defined as follows:  $((x_1;\xi_1),(x_2;\xi_2))$ is \emph{not} in \eqref{eq:wfs} if and only if for all $\varepsilon>0$ there exists a uniformly bounded  family $B_i(z)\in \Psi^{0}(M)$ of  properly supported operators, each elliptic at $(x_i;\xi_i)$ and such that for all $r\in \rr$, the family
$$
\bra z\ket^{\12}B_1(z) G(z) B_2(z)^* \mbox{ for } z\in Z \mbox{ is bounded in } B(H^{r}_\c(M), H_\loc^{r+s}(M)).
$$
\end{definition}

We define the \emph{uniform operator wavefront set of order $s\in\rr$ and weight $1$} in the same way,  with $\bra z\ket^{\12}$ replaced by $1$, and we denote that set $\wfl{0}\big( G(z) \big)$ for simplicity.  Definition \ref{defrrr} is  similar to the definition from  \cite[\S3]{Dang2020}, with the only difference  that we allow $B_i$ to depend on $z$ (which is easier to verify in practice).


Let us denote by $\Delta^*$ be the diagonal in $(T^*M \setminus \zero)^{\times 2}$, i.e.
$$
\Delta^* = \{ ((x_1;\xi_1),(x_2;\xi_2))  \, | \,    x_1=x_2, \ \xi_1=\xi_2 \}   \subset (T^*M \setminus \zero)^{\times 2}.
$$

\begin{definition}\label{deffey} 
The  \emph{Feynman wavefront set} $\Lambda\subset (T^*M\setminus\zero)^{\times 2}$ is defined by
$$
\bea
 \Lambda & \defeq   \big( (\Sigma^+)^{\times 2} \cap   \{  ((x_1;\xi_1),\! (x_2;\xi_2))\, | \, (x_1;\xi_1)  \sim  (x_2;\xi_2) \mbox{ and } x_1\in J_-(x_2) \}\big)  \fantom \ \,  \cup  \big( (\Sigma^-)^{\times 2} \cap   \{  ((x_1;\xi_1),\! (x_2;\xi_2))\, | \, (x_1;\xi_1)  \sim  (x_2;\xi_2) \mbox{ and } x_1\in J_+(x_2) \} \big) \! \cup \Delta^*.
 \eea
$$
\end{definition}

In the definition we employed the convention  which corresponds to considering  \emph{primed} wavefront sets (as opposed to wavefront sets of Schwartz kernels). We caution the reader that beside the choice of  working with `primed' or 'non-primed' wavefront sets, in the context of QFT there are two sign conventions possible.

As in \cite{Dang2020} we will use the following version of Hörmander's propagation of singularities theorem, formulated in terms of the uniform wavefront set.
 
 \begin{proposition}\label{opprop} Let $Z\subset \{ z \in \cc \st \Im  z \geqslant  0\}$. Suppose that for all $m\in\rr$, $G(z)$ and  $(P-z)G(z)$ are bounded families of operators in  $B(H^m_\c(M),H^m_\loc(M))$ for $z\in Z$. Suppose
 \beq\label{eqt1}
 ((x_1;\xi_1),\!(x_2;\xi_2))\in \wfl{0}\big(G(z)\big)\setminus \wfl[s-1]{0}\big((P-z)G(z)\big).
 \eeq
Then $(x_1;\xi_1)\in \Sigma$.   Furthermore,  $((x'_1;\xi'_1),\!(x_2;\xi_2))\in\wfl{0}\big( G(z)\big)$ for all $(x'_1;\xi'_1)$ such that $(x'_1;\xi'_1)\sim (x_1;\xi_1)$ and $(x_1';\xi'_1)$ precedes $(x_1;\xi_1)$ along the bicharacteristic flow, provided that $((x;\xi),\!(x_2;\xi_2))\notin\wfl[s-1]{0}\big((P-z)G(z)\big)$ for all $(x;\xi)$ on the bicharacteristic connecting $(x_1;\xi_1)$ and $(x'_1;\xi'_1)$.
 \end{proposition}
 \begin{proof}  We explain the relationship to better known formulations for the sake of completeness, see \cite{Dang2020} for more details. In what follows, all pseudo-differential operators are assumed compactly supported.
 
  The proof of  propagation of singularities by positive commutator arguments \cite{Hormander1971} gives a uniform estimate of the following form.  Let  $s\in\rr$, $N\ll 0$. For any $B_1'\in\Psi^0(M)$ elliptic at $(x_1;\xi_1)$, and any $B\in \Psi^0(M)$ elliptic in a neighborhood of the bicharacteristic from $(x_1';\xi_1')$ to  $(x;\xi)$, we have
 \beq\label{est1}
 \|  B_1 u\|_{s}  \leqslant C (  \|  B_1' u\|_{s}   + \|  B (P-z) u \|_{s-1} + \norm{\chi u}_{N})
 \eeq
  uniformly for $u\in H^{N}_{\rm loc}(M)$ and $z\in Z$, where  $B_1\in \Psi^0(M)$ is some $\Psi$DO elliptic  at $(x_1;\xi_1)$ and $\chi \in \ccf$. Now, suppose $((x_1';\xi_1'),\!(x_2;\xi_2))\notin \wfl{0}(G(z))$. Then there exist $B_1'(z),B_2(z)\in\Psi^0(M)$ elliptic at respectively $(x_1';\xi_1'),(x_2;\xi_2)$ such that for any bounded subset $\cU\subset H^l_\c(M)$, the set $B_1'(z) G(z) B_2^*(z) \cU$ is uniformly bounded in $H^{l+s}_\loc(M)$. By \eqref{est1} applied to  elements of $G(z) B_2^* \cU$,  $B_1 G(z) B_2^*(z) \cU$ is bounded in $H^{l+s}_\loc(M)$, hence $ ((x_1;\xi_1),\!(x_2;\xi_2))\notin\wfl{0}(G_z)$.
 \end{proof}

 Note that $\wfl{0}(\one)=\Delta^*$ for large $s\in\rr$. Thus, if  $(P-z) G(z)=\one$, then Proposition \ref{opprop} says  that we can propagate singularities (or equivalently, regularity) of $G(z)$ along bicharacteristics in the first factor as long as they do not hit $\Delta^*$.

 There is an analogous statement for propagation in the second factor of $(T^*M\setminus\zero)^{\times 2}$ if $G(z)(P-z)$ is bounded in $B(H^m_\c(M),H^m_\loc(M))$. Namely, if
  \beq\label{eqt2}
  ((x_1;\xi_1),\!(x_2;\xi_2))\in \wfl{0}\big(G(z)\big)\setminus \wfl[s-1]{0}\big(G(z)(P-z)\big),
  \eeq
 then $(x_2;\xi_2)\in \Sigma$. Furthermore,  $((x_1;\xi_1),\!(x_2';\xi_2'))\in \wfl{0}(G(z))$   for all $(x_2';\xi_2')$ such that $(x_2';\xi_2')\sim (x_2;\xi_2)$, provided that $((x_1;\xi_1),\!(x;\xi))\notin\wfl[s-1]{0}\big(G(z)(P-z)\big)$ for all $(x;\xi)$ on the bicharacteristic connecting $(x_2;\xi_2)$ and $(x_2';\xi_2')$.

For $\varepsilon>0$, let $Z_\varepsilon\subset \cc$ be a ``punctured sector'' in the upper half-plane of the form
\beq 
Z_\varepsilon := \{ z \in \cc \,\st\,  \module{\arg  z - \pi/2} < \theta , \ \module{z}\geq \varepsilon\} 
\eeq
for some arbitrarily chosen $\theta\in \open{0,\pi/2}$.

\begin{proposition} If $Z=Z_\varepsilon$ with $\varepsilon >0$ then in Proposition \ref{opprop} we can replace $((x'_1;\xi'_1),\!(x_2;\xi_2))\in\wfl{0}\big( G(z)\big)$ by $((x'_1;\xi'_1),\!(x_2;\xi_2))\in\wfl[s-1/2]{12}\big( G(z)\big)$
\end{proposition}
\begin{proof}  The positive commutator argument   used to prove \eqref{est1} gives  actually the stronger estimate
 \beq\label{est2}
 \|  B_1 u\|_{s}  + (\Im z)^\12  \|  B_1 u\|_{s-\12} \leqslant C (  \|  B_1' u\|_{s}   + \|  B (P-z) u \|_{s-1} + \norm{\chi u}_{N}),
 \eeq
  see \cite{Dang2020} for more details.
Furthermore,   $$ \|  B_1 u\|_{s-\12}  \leq C_1 \bra z \ket^{-\12} (\|  B_1 u\|_{s}  + (\Im z)^\12  \|  B_1 u\|_{s-\12})$$ for some $C_1>0$ uniformly in $z\in  Z_\varepsilon$. Hence,
$$
\|  B_1 u\|_{s-\12}   \leqslant C_2  \bra z \ket^{-\12}   (  \|  B_1' u\|_{s}   + \|  B (P-z) u \|_{s-1} + \norm{\chi u}_{N}),
$$
and from that point on we can apply the argument recalled after \eqref{est1}.  
\end{proof}

\subsection{Uniform resolvent estimate} We first prove a basic estimate on regularity properties of $(P-z)^{-1}$, which later on enables us to use the operator formulation of propagation of singularities.

\begin{lemma} For all $m\geqslant 0$,  the family of operators $(P-z)^{-1}$, $\Im z >0$, is bounded in $B(H_\c^m(M),H_\loc^{m+1}(M))$.
\end{lemma}
\begin{proof} This can be shown in a similar vein as Proposition \ref{key1}.   Namely, let $f\in H_\c^m(M)$. By Lemma \ref{lelemme},  for every  $(x_1;\xi_1)=(t_1,y_1;\tau_1,\eta_1)\in \Sigma^\pm$ with  $\pm t_1$ sufficiently large there exists a bounded family  $B_\pm(z)\in \Psi^0(M)$ such that $B_\pm(z)$ is elliptic at $(x_1;\xi_1)$ and
\beq\label{rtt}
B_\pm(z) (P-z)^{-1}f = 0,
\eeq
hence  $(x_1;\xi_1)\notin \wf^{(m+1)}((P-z)^{-1} f)$ by \eqref{rtt}. Since $\wf^{(m)}(f)=\emptyset$, by propagation of singularities  applied to $(P-z)^{-1}f$ we get $(x;\xi)\notin  \wf^{(m+1)}((P-z)^{-1} f)$ for all $(x;\xi)$ such that $(x;\xi)\sim (x_1;\xi_1)$ and $\pm t \leq \pm t_1$. In conclusion, $\Sigma^\pm \cap \wf^{(m+1)}((P-z)^{-1} f)= \emptyset$. On the other hand $\wf^{(m+1)}((P-z)^{-1} f) \subset \Sigma = \Sigma^+ \cup \Sigma^-$ by elliptic regularity. Hence  $\wf^{(m+1)}((P-z)^{-1} f)=\emptyset$, which yields $(P-z)^{-1} f\in H^{m+1}_\loc(M)$. By the uniformity of propagation estimates and of the elliptic estimate, $H^{m+1}_\loc(M)$-seminorms of $(P-z)^{-1} f$ are bounded by $H^{m}_\c(M)$-seminorms of $f$, uniformly in $z$. 
\end{proof}


We are now ready to prove that the uniform operator wavefront set of $(P-z)^{-1}$ in $Z_\varepsilon$ is contained in the Feynman wavefront $\Lambda$.

\medskip 

\refproof{Theorem \ref{theorem2}} \step{1}  Let $(x_1;\xi_1)=(t_1,y_1;\tau_1,\eta_1)\in \Sigma^\pm$ with  $\pm t_1>T$ (where $T$ is as in \sec{ss:locreg}) and let $(x_2;\xi_2)=(t_2,y_2;\tau_2,\eta_2)\in T^*M \setminus \zero$ be such that $\pm t_1 >\pm t_2$. Then by Lemma \ref{lelemme}, there exists a bounded family  $B_\pm(z)\in \Psi^0(M)$ such that $B_\pm(z)$ is elliptic at $(x_1;\xi_1)$ and
$$
B_\pm(z) \circ    (P-z)^{-1} \circ \chi  = 0
$$ 
for some $\chi\in \ccf$ with   $\chi(x_2)\neq 0$, provided that $\supp \chi$ is a sufficiently small neighborhood of $x_2$. This implies $((x_1;\xi_1),(x_2;\xi_2))\notin \wfl{0}\big(   (P-z)^{-1}   \big)$ for all $s\in\rr$.  In conclusion,
\beq\label{rtrt1}
\big( \Sigma^\pm\times (T^*M\setminus\zero) \big)\cap \{ \pm t_1 > T, \ \pm t_1 > \pm t_2\}   \cap \wfl{0}  \big( ( P-z)^{-1} \big) = \emptyset.
\eeq

\step{2} Next, we use propagation of singularities to deduce
\beq\label{rtrt2}
\big( \Sigma^\pm\times (T^*M\setminus\zero) \big)  \cap \wfl{12}  \big( ( P-z)^{-1} \big) \subset \Lambda.
\eeq
More precisely, let $(x;\xi)\in \Sigma^\pm$ and suppose $(x_2;\xi_2)\in T^*M \setminus\zero$ is such that
\beq\label{eq:tco}
((x;\xi),\!(x_2;\xi_2)) \in  \wfl{12}  \big( ( P-z)^{-1} \big)\setminus \Lambda.
\eeq
Since $(x;\xi)\in \Sigma^\pm$ and $((x;\xi),\!(x_2;\xi_2))  \notin \Lambda$, we can find $(x_1;\xi_1)=(t_1,y_1;\tau_1,\eta_1)\in \Sigma^\pm$ with $\pm t_1>\max\{T, \pm t_2\}$ such that $(x_1;\xi_1)\sim (x;\xi)$  and  $(x_2;\xi_2)$ does not intersect the bicharacteristic connecting $(x_1;\xi_1)$ and $(x;\xi)$. By  \eqref{rtrt1},  $((x_1;\xi_1),\!(x_2;\xi_2))\notin  \wfl{0}  \big( ( P-z)^{-1} \big)$. By propagation of singularities in the form given in Proposition \ref{opprop} this implies  $((x;\xi),\!(x_2;\xi_2))\notin  \wfl{12}  \big( ( P-z)^{-1} \big)$, which contradicts \eqref{eq:tco}. The argument is valid for any $(x;\xi)\in \Sigma^\pm$, so we  conclude  \eqref{rtrt2}.

\step{3}  By proceeding analogously in the second factor, we obtain 
\beq\label{rtrt3}
\big( (T^*M\setminus\zero)\times \Sigma^\pm \big)  \cap \wfl{12}  \big( ( P-z)^{-1} \big) \subset \Lambda.
\eeq
In combination with the two versions of identity \eqref{rtrt2} this yields
\beq \label{yyy}
(  \Sigma\times \Sigma)  \cap \wfl{12}  \big( ( P-z)^{-1} \big) \subset \Lambda.
\eeq
On the other hand, by the elliptic regularity statement in Proposition \ref{opprop} and its analogue in the second factor, we have
$$
   \wfl{12}   \big( ( P-z)^{-1} \big)  \subset  (\Sigma \times \Sigma) \cup \Delta^*.
$$
Thus \eqref{yyy} implies 
$
   \wfl{12}   \big( ( P-z)^{-1} \big)  \subset    \Lambda \cup \Delta^* = \Lambda,
$
which concludes the proof. \qed

\medskip

{\small
\subsubsection*{Acknowledgments} The authors would like to thank Nguyen Viet Dang and Kouichi Taira for stimulating discussions.  The research leading to these results has received funding from the European Union's Horizon 2020 research and innovation programme under the Marie Sk{\l}odowska-Curie grant agreement No 754340.  Support from the grant
ANR-20-CE40-0018 is gratefully acknowledged. \medskip }

\bibliographystyle{abbrv}
\bibliography{complexpowers}

\end{document}